\documentclass[letter,10pt]{article}
\usepackage[margin=0.1in,footskip=0.2in]{geometry}
\usepackage[english]{babel}
\usepackage{graphicx}
\usepackage{amssymb}
\usepackage{amsthm}
\usepackage{amssymb}
\usepackage{setspace}
\usepackage[normalem]{ulem}

\usepackage{amsmath,amsthm,amscd,amsfonts,mathrsfs,amssymb}
\usepackage{graphicx,enumerate}

\usepackage[all]{xy}
\usepackage{fullpage} 
\usepackage{color}
\usepackage{makeidx}
\usepackage{alltt}
\usepackage{setspace}
\usepackage{hyperref}

\newtheorem{theorem}{Theorem}[section]
\newtheorem{lemma}[theorem]{Lemma}

\newtheorem{proposition}[theorem]{Proposition}
\theoremstyle{definition}
\newtheorem{remark}[theorem]{Remark}
\theoremstyle{definition}
\newtheorem{definition}[theorem]{Definition}
\newtheorem{example}[theorem]{Example}

\newcommand{\ssl}[2]{\textrm{SL}(#1,\mathbb{#2})}
\DeclareMathOperator{\GL}{GL}

\newcommand{\Z}{ \mathbb{Z}}
\newcommand{\C}{\mathbb{C}}
\newcommand{\reff}[1]{(\ref{#1})}
\newcommand{\bra}[1]{\left(#1\right)}

\renewcommand{\d}{\mathrm{d}}

\newcommand{\chibs}{\overline{\chi^*}}
\newcommand{\chis}{\chi^*}

\newcommand{\aquad}{\qquad\qquad}
\newcommand{\bquad}{\aquad\aquad}
\newcommand{\cquad}{\bquad\bquad}
\newcommand{\inv}{^{-1}}

\DeclareMathOperator{\Kl}{Kl}
\newcommand{\sums}{\sideset{}{^*}\sum}
\newcommand{\nn}{h}

\newcommand{\lt}{\left}
\newcommand{\rt}{\right}
\renewcommand{\vec}{\mathbf}
\renewcommand{\mod}{\textrm{ mod }}
\newcommand{\hhh}{\mathcal{H}}
\renewcommand{\ggg}{\mathcal{G}}
\newcommand{\TT}{T}

\begin{document}

\title{\scshape The Voronoi formula and double Dirichlet series}

\author{Eren Mehmet K{\i}ral and Fan Zhou}

\maketitle
\begin{abstract} 
We prove  a Voronoi formula for coefficients of a large class of $L$-functions including Maass cusp forms, Rankin-Selberg convolutions, and certain non-cuspidal forms. Our proof is based on the functional equations of $L$-functions twisted by Dirichlet characters and does not directly depend on automorphy. Hence it has wider application than previous proofs. The key ingredient  is the construction of a double Dirichlet series.
\end{abstract}
\section{Introduction}

A Voronoi formula is an identity involving Fourier coefficients of automorphic forms, with the coefficients  twisted by additive characters on either side.  A history of the Voronoi formula can be found in \cite{millerschmidHistory}. Since its introduction in \cite{millerschmid1}, the Voronoi formula on $\GL(3)$ of Miller and Schmid  has become a standard tool in the study of $L$-functions arising from $\GL(3)$, and has found important applications such as  \cite{blomerbuttcane}, \cite{blomerKhanYoung}, \cite{khan}, \cite{licentral}, \cite{lisubconvexity},  \cite{liyoung}, \cite{miller}, \cite{munshiShiftedConv} and \cite{munshi4}. 
As of yet the general $\GL(N)$ formula has had fewer applications, a notable one being \cite{kowalskiGuillaume}.

The first proof of a Voronoi formula on $\GL(3)$ was found by Miller and Schmid in \cite{millerschmid1} using the theory of automorphic distributions.
Later, a Voronoi formula was established for GL($N$) with $N\geq 4$ in  \cite{goldfeldli1}, \cite{goldfeldli2}, and \cite{millerschmid2}, with \cite{millerschmid2} being more general and earlier than \cite{goldfeldli2} (see the addendum, loc.\! cit.\!). Goldfeld and Li's proof
\cite{goldfeldli2}
 is more akin to the classical proof in $\GL(2)$ (see \cite{good}), obtaining the associated Dirichlet series through a shifted ``vertical'' period integral and making use of automorphy. 
An adelic version was established by Ichino and Templier in \cite{ichinotemplier}, 
allowing ramifications and applications to number fields. Another direction of generalization with more complicated additive twists on either side has been considered in an unpublished work of Li and Miller and in  \cite{zhou}.

In this article, we prove a Voronoi formula for a large class of automorphic objects or $L$-functions, including cusp forms for $\ssl{N}{Z}$, Rankin-Selberg convolutions, and certain non-cuspidal forms.
Previous works (\cite{millerschmid2}, \cite{goldfeldli2}, \cite{ichinotemplier}) do not offer a Voronoi formula for Rankin-Selberg convolutions or non-cuspidal forms.
Even for Maass cusp forms, our new proof is shorter than any previous one, and uses a completely different set of techniques.

Let us briefly summarize our method of proof. 
We first reduce the statement of Voronoi formula to a formula involving Gauss sums of Dirichlet characters. 
We construct a complex function of two variables and write it as double Dirichlet series in two different ways by applying a functional equation. Using the uniqueness theorem of Dirichlet series, we get an identity between coefficients of these two double Dirichlet series.  This leads us to the Voronoi formula with Gauss sums.

One of our key steps in obtaining the Voronoi formula is the use of functional equations of $L$-functions twisted by Dirichlet characters. 
The relationship between the Voronoi formulas and the functional equations of these $L$-functions is known from previous works, such as  \cite{dukeiwaniec}, Section 4 of \cite{goldfeldli1}, \cite{buttcane} and \cite{zhou}. 
Miller-Schmid derived the functional equation of $L$-functions twisted by a Dirichlet character of prime conductor from the Voronoi formula in Section 6 of \cite{millerschmid1}.
However there is a combinatorial difficulty in reversing this process, i.e., obtaining additive 
 twists of general non-prime conductors from multiplicative ones, which was acknowledged in both  \cite[p. 430]{millerschmid1}  and \cite[p. 68]{ichinotemplier}. 
The method presented here is able to overcome this difficulty by discovering an interlocking structure among a family of Voronoi formulas with different conductors.

Our proof of the Voronoi formula is complete for additive twists of all conductors, prime or not, and unlike \cite{goldfeldli1}, \cite{goldfeldli2},\cite{ichinotemplier}, \cite{millerschmid1}, or  \cite{millerschmid2}, does not depend directly on automorphy of the cusp forms.  
This fact allows us to apply our theorem to many conjectural Langlands functorial transfers. For example, the Rankin-Selberg convolutions (also called functorial products) for $\GL(m)\times \GL(n)$ are not yet known to be automorphic on $\GL(m \times n)$ in general.
Yet  we know the functional equations of $\GL(m) \times \GL(n)$ $L$-functions twisted by Dirichlet characters. Thus, our proof provides a Voronoi formula for the Rankin-Selberg convolutions on $\GL(m) \times \GL(n)$ (see Example \ref{ex:rankinSelberg}). 
Voronoi formulas for these functorial cases are unavailable from \cite{goldfeldli2}, \cite{millerschmid2} or \cite{ichinotemplier}. 
In Theorem \ref{conversethm} we reformulate our Voronoi formula like the classical converse theorem of Weil, i.e., assuming every $L$-function twisted by Dirichlet character is entire, has an Euler product (or satisfies Hecke relations), and satisfies the precise functional equations, then the Voronoi formula as in Theorem \ref{thm:glnVoronoi} is valid. We do not have to assume it is a standard $L$-function coming from a cusp form.

Furthermore, by Theorem \ref{conversethm}, we obtain a Voronoi formula for certain non-cuspidal forms, such as isobaric sums (see Example \ref{ex:isobaric}). This is not readily available from any previous work but
it is believed (see \cite{millerschmid2} p.\!\! 176)  that one may  derive a formula by using formulas on smaller groups through a possibly complicated procedure. 
Such complication does not occur in our method because we work directly with $L$-functions.
\\

We first state the main results for Maass cusp forms.
Denote $e(x):=\exp(2\pi i x)$ for $x\in \mathbb{R}$.
Let $N\geq 3$ be an integer.
Let $a,n \in \Z$,  $c\in \mathbb{N}$ and let
	$$
		\vec{q}=  (q_1,q_2, \ldots, q_{N-2}) \qquad \text{ and } \qquad 
		\vec{d} = (d_1, d_2, \ldots, d_{N-2})
	$$
be tuples of positive integers satisfying the divisibility conditions 
\begin{equation}\label{eq:divisibilityConditions}
d_1|q_1c,\quad d_2\left|\frac{q_1q_2c}{d_1},\qquad \cdots,\qquad 
d_{N-2}\right|\frac{q_1\cdots q_{N-2}c}{d_1 \cdots d_{N-3}}.
\end{equation}

Define the hyper-Kloosterman sum as
	\begin{align*}
		\Kl(a,n,c;\vec{q},\vec{d}) =& \sums_{x_1 (\text{mod } \frac{q_1c}{d_1})} \,\,\,\,\,
		\sums_{x_2 (\text{mod } \frac{q_1q_2c}{d_1d_2})} \cdots 
		\sums_{x_{N-2} (\text{mod } \frac{q_1\cdots q_{N-2} c}{d_1 \cdots d_{N-2} })}\\
		&\times
		e\lt(\frac{d_1x_1a}{c} + \frac{d_2x_2 \overline{x_1}}{\frac{q_1c}{d_1}}
		+ \cdots + 
		\frac{d_{N-2}x_{N-2} \overline{x_{N-3}}}{\frac{q_1\cdots q_{N-3} c}{d_1 \cdots d_{N-3} }}
		+ \frac{n \overline{x_{N-2}}}{\frac{q_1\cdots q_{N-2} c}{d_1 \cdots d_{N-2} }}\rt),
	\end{align*}
	where $\sums$ indicates that the summation is over reduced residue classes, and 
	$\overline{x_i}$ denotes the multiplicative inverse of $x_i$ modulo $q_1\cdots q_i c/d_1 \cdots d_i$. 
	When $N=3$, $\Kl(a,n,c;q_1,d_1)$ becomes the classical Kloosterman sum 
	$S(aq_1,n ;cq_1/d_1)$. 
For the degenerate case of  $N=2$, we define $\Kl(a,n,c;\;,\;):=e(an/c)$.


Let $F$ be a Hecke-Maass cusp form for $\ssl{N}{Z}$ with the spectral parameters $(\lambda_1,\dots,\lambda_n)\in \mathbb{C}^n$. Let $A(*,\ldots,*)$ be the Fourier-Whittaker coefficients of $F$ normalized as $A(1,\ldots,1)=1.$
 We refer to \cite{goldfeld} for the definitions and the basic results of Maass forms for $\ssl{N}{Z}$. The Fourier coefficients satisfy the Hecke relations
\begin{equation}\label{glnheckerelation0}
A(m_1m_1', \cdots, m_{N-1}m_{N-1}') = A(m_1,\ldots,m_{N-1})A(m_1', \ldots, m_{N-1}')
\end{equation}
if $(m_1\cdots m_{N-1}, m_1' \cdots m_{N-1}') = 1$ is satisfied,
\begin{equation}\label{glnheckerelation1}
A(1,\dots,1,n)A(m_{N-1},\dots,m_1)=
\sum_{\substack{d_0\dots d_{N-1}=n\\d_1|m_1,\dots,d_{N-1}|m_{N-1}}}A\bra{\frac{m_{N-1}d_{N-2}}{d_{N-1}},\dots,  \frac{m_2d_1}{d_2} , \frac{m_1d_0}{d_1}},
\end{equation}
and
\begin{equation}\label{glnheckerelation2}
A(n,1,\dots,1)A(m_1,\dots,m_{N-1})=\sum_{\substack{d_0\dots d_{N-1}=n\\d_1|m_1,\dots,d_{N-1}|m_{N-1}}}A\bra{\frac{m_1d_0}{d_1},\frac{m_2d_1}{d_2},\dots,\frac{m_{N-1}d_{N-2}}{d_{N-1}}}.
\end{equation}

The dual Maass form of $F$ is denoted by $\widetilde{F}$. Let $B(*,\dots,*)$ be the Fourier-Whittaker coefficients of $\widetilde F$. These coefficients satisfy
\begin{equation}\label{eq:dualCoefficients}
B(m_1,\dots,m_{N-1})=A(m_{N-1},\dots,m_1).
\end{equation}

Define the ratio of Gamma factors
\begin{equation}\label{eq:plusminusGammaFactors}
	G_\pm(s):= {i^{-N\delta}}\pi^{-N(1/2-s)}\prod_{j=1}^N \Gamma\bra{\frac{\delta + 1-s-\overline{\lambda_j}}{2}}\Gamma\bra{\frac{\delta + s-{\lambda_j}}{2}}^{-1},
\end{equation}
where for even Maass forms, we define $\delta = 0$ in $G_+$ and $\delta  = 1 $ in $G_-$ and for odd Maass forms, we define $\delta = 1$ in $G_+$ and $\delta  = 0 $ in $G_-$. 	
We refer to Section 9.2 
of \cite{goldfeld} for the definition of even 
and odd Maass forms.

\begin{theorem}[Voronoi formula on GL($N$) of Miller-Schmid \cite{millerschmid2}]\label{thm:glnVoronoi}
Let $F$ be a Hecke-Maass cusp form with coefficients $A(*,\ldots, *)$, and $G_\pm$ a ratio of Gamma factors as in \eqref{eq:plusminusGammaFactors}. Let $c>0$ be an integer and let $a$ be any integer with $(a,c)=1$. Denote by $\overline{a}$ the multiplicative inverse of $a$ modulo $c$. 
Let the additively twisted Dirichlet series be given as 
\begin{equation}\label{eq:LqTwistedAdditively}
L_{\vec{q}}( s , F , a/c) = \sum_{n=1}^\infty \frac{A(q_{N-2},\dots, q_1,n)}{n^s}
e\bra{\frac{\overline{a}n}{c}}	
\end{equation}
for $\Re(s)>1$. 
	This Dirichlet series has an analytic continuation to all $s \in \C$ and satisfies the functional
	equation 
	\begin{align}\label{eq:glnVoronoi}
	\begin{split}
		L_{\vec{q}}(s,F,a/c) =& \frac{ G_+(s) - G_{-}(s) }{2} \sum_{d_1|q_1c}\sum_{d_2 | \frac{q_1q_2c}{d_1}} 
		\cdots \sum_{d_{N-2}| \frac{q_1\ldots q_{N-2} c}{d_1 \ldots d_{N-3}}}
		\\
		&\quad \times \sum_{n=1}^\infty 
		\frac{A(n,d_{N-2}, \ldots, d_2, d_1) \Kl(a,n, c; \vec{q}, \vec{d}) }
		{n^{1-s} c^{Ns-1}d_1 d_2 \cdots d_{N-2}} 
		\frac{d_1^{(N-1)s}d_2^{(N-2)s}\cdots d_{N-2}^{2s}}
		{q_1^{(N-2)s}q_2^{(N-3)s}\cdots q_{N-2}^{s}} 
		\\
		&+ \frac{G_+(s) + G_{-}(s)	}{2} 
\sum_{d_1|q_1c}\sum_{d_2 | \frac{q_1q_2c}{d_1}} 
		\cdots \sum_{d_{N-2}| \frac{q_1\ldots q_{N-2} c}{d_1 \ldots d_{N-3}}}		
		\\
		&\quad \times \sum_{n=1}^\infty 
		\frac{A(n,d_{N-2}, \ldots, d_2, d_1) \Kl(a,-n, c; \vec{q}, \vec{d}) }
		{n^{1-s} c^{Ns-1}d_1 d_2 \cdots d_{N-2}} 
		\frac{d_1^{(N-1)s}d_2^{(N-2)s}\cdots d_{N-2}^{2s}}
		{q_1^{(N-2)s}q_2^{(N-3)s}\cdots q_{N-2}^{s}},
		\end{split}
	\end{align}
	in the region of convergence of the expression on the right hand side ($\Re(s)<0$).
\end{theorem}

The traditional Voronoi formula, involving weight functions instead of Dirichlet series, is obtained after taking an inverse Mellin transform against a suitable test function.

Choose a Dirichlet character $\chi$ modulo $c$, which is not necessarily primitive, 
multiply both sides of \eqref{eq:glnVoronoi} by $\chi(a)$, and sum this equality over the reduced residue system modulo $c$. 
We obtain the following the Voronoi formula with Gauss sums. 
In Section \ref{subsec:equivalent} we show through elementary finite arithmetic that the formulas \eqref{eq:glnVoronoi} and \eqref{eq:HequalsG} are equivalent.

\begin{theorem}[Voronoi formula with Gauss sums]\label{thm:glnAverageVoronoi}
Let $\chi$ be a Dirichlet character modulo $c$, induced from the primitive character
	$\chi^*$ modulo $c^*$ with $c^*|c$. 
Define for $\vec{q}=(q_1,\dots,q_{N-2})$ a tuple of positive integers
	\begin{equation}\label{eq:Hdefn}
		H(\vec{q}; c, \chi^*, s)= \sum_{n=1}^\infty 
		\frac{A(q_{N-2}, \ldots, q_{1}, n) g(\overline{\chi^*},c,n)}{n^s (c/c^*)^{1-2s}},
	\end{equation}
for $\Re(s)>1$ and
	\begin{align}
		G&(\vec{q}; c, \chi^*,s) =\frac{G(s) \chis(-1)}{c^{Ns-1}(c/c^*)^{1-2s}}
		\sum_{d_1c^*|q_1c}\sum_{d_2 c^*| \frac{q_1q_2c}{d_1}} 
		\cdots \sum_{d_{N-2}c^*| \frac{q_1\ldots q_{N-2} c}{d_1 \ldots d_{N-3}}}		\notag
\\
		&\times \sum_{n=1}^\infty 
		\frac{A(n,d_{N-2}, \ldots, d_1)}
		{n^{1-s}d_1 d_2 \cdots d_{N-2}}
		\frac{d_1^{(N-1)s}d_2^{(N-2)s}\cdots d_{N-2}^{2s}}
		{q_1^{(N-2)s}q_2^{(N-3)s}\cdots q_{N-2}^{s}}\label{eq:Gdefn}\\
		&\times g(\chi^*, c,d_1) g(\chi^*, \tfrac{q_1c}{d_1},d_2) 
		\cdots g(\chi^*, \tfrac{q_1\cdots q_{N-3}c}{d_1 \cdots d_{N-3}}, d_{N-2})
		g(\chi^*, \tfrac{q_1\cdots q_{N-2}c}{d_1\cdots d_{N-2}},n) \notag
	\end{align}
for $\Re(s)<0$,	
where $G$ equals $G_+$ or $G_-$ depending on whether $\chis(-1)$ is $1$ or $-1$, 
and $g(\chis,\ell c^*,*)$ is the Gauss sum of the induced character modulo $\ell c^*$ from $\chis$, which is defined in Definition \ref{definegausssum}.
Both functions have analytic continuation to all $s \in \C$, and the equality 
\begin{equation}\label{eq:HequalsG}
	H(\vec{q}; c,\chi^*,s) = G(\vec{q}; c,\chi^*,s)
\end{equation} is satisfied.
\end{theorem}

In proving \eqref{eq:HequalsG}, we define
\begin{equation}\label{eq:zsw}
	Z(s,w) = \frac{L_{\vec{q}}(2w-s,F)L(s,F\times \chi^*)}{L(2w-2s+1,\overline{\chi^*})},
\end{equation}
where $\vec{q} = (q_1, \ldots, q_{N-2})$ is a tuple of positive integers, and the function $L_{\vec{q}} (s,F)$ is given as the Dirichlet series
\[
	L_{\vec{q}}(s,F)= \sum_{n=1}^\infty \frac{A(q_{N-2}, \ldots, q_1, n) }{n^s},
\]
for $\Re(s)\gg 1$.  We express $Z(s,w)$ as a double Dirichlet series in two different ways. In one region of convergence we express the $L$-functions as Dirichlet series and obtain $$Z(s,w) = \sum_{n=1}^\infty \frac{a_n(s)}{n^{2w}}.$$
On the other hand we apply the functional equation of $L(s,F\times \chi^*)$, replacing $s$ with $1-s$, and 
write $Z(s,w)$ as the Dirichlet series $$Z(s,w) = \sum_{n} \frac{b_n(s)}{n^{2w}}.$$ By the uniqueness of Dirichlet series, we must have $a_n(s)=b_n(s)$. This equality leads us to the Voronoi formula with Gauss sums.


Our proof only uses the Hecke relations about the Fourier coefficients of $F$ and the exact form of the functional equations. The expression of Gamma factors, or the automorphy of $F$, plays no role. Hence we can formulate our theorem in a style similar to the classical converse theorem of Weil. First let us list the properties of Fourier coefficients that we use in order to state the following theorem.


The Fourier coefficients of $F$ grow moderately, i.e.,
\begin{equation}\label{eq:moderateGrowth}
 A(m_1,\dots,m_{N-1}) \ll \bra{m_1\dots m_{N-1}}^{\sigma}
\end{equation}
for some $\sigma >0$. Given a primitive Dirichlet character $\chis$ modulo $c^*$,
define the twisted $L$-function 
\begin{equation}\label{eq:twistedLfunction}
	L(s, F\times \chis)=\sum_{n=1}^\infty \frac{A(1,\dots,1,n)\chis(n)}{n^s},
\end{equation}
for $\Re(s)>\sigma+1$. It has analytic continuation to the whole complex plane, and satisfies the functional equation
\begin{equation}\label{gln-FE-character}
L(s, F\times \chis)= \tau(\chis)^N {c^*}^{-Ns}G(s) L(1-s, \widetilde{F}\times \chibs),
\end{equation}
where $G(s)=G_+(s)$ or $G_-(s)$ depending on whether $\chis(-1)=1$ or $-1$.

%


\begin{theorem}\label{conversethm}
Let $F$ be a symbol and assume that with $F$ comes numbers $A(m_1,\ldots, m_{N-1})\in \mathbb{C}$ attached to every $(N-1)$-tuple $(m_1,\ldots,m_{N-1})$ of natural numbers. Assume $A(1,\ldots,1)=1$.

Assume that these ``coefficients'' $A(*,\ldots,*)$ satisfy the aforementioned 
Hecke relations \eqref{glnheckerelation0}, \eqref{glnheckerelation1} and \eqref{glnheckerelation2}.
Further assume that they grow moderately as in \eqref{eq:moderateGrowth}.

Let $\widetilde{F}$ be another symbol whose associated coefficients 
$B(*,\ldots,*)\in \mathbb{C}$
are given as in \eqref{eq:dualCoefficients}
and assume that they also satisfy the same properties. 
Further, assume that there are two meromorphic functions  $G_{+}(s)$ and $G_{-}(s)$ associated to the pair $(F,\widetilde F)$, so that for a given primitive character $\chis$, the function $L(s,F\times\chis)$ as defined in \eqref{eq:twistedLfunction} satisfies the functional equation \eqref{gln-FE-character}.
 
Under all these assumptions, $L_\vec{q}(s,F, a/c)$ defined as in \eqref{eq:LqTwistedAdditively} for $\Re(s)>1+\sigma$,  has analytic continuation to all $s \in \C$, and satisfies the Voronoi formula \eqref{eq:glnVoronoi}. (The Dirichlet series on the right side of \eqref{eq:glnVoronoi} is absolutely convergent for $\Re(s)<-\sigma$.)

Equivalently the functions $H(\vec{q}; c,\chis, s)$ and $G(\vec{q};c,\chis,s)$ as defined by the formulas \eqref{eq:Hdefn} and \eqref{eq:Gdefn} have analytic continuations to all $s$ and equal each other as in \eqref{eq:HequalsG}.
\end{theorem}

\begin{remark}[The structure of this article]
Theorem \ref{conversethm} is our main result. 
For the most part our focus is $N\geq 3$, and we deal with the case $N=2$ in Remark \ref{remark-GL2}.
The Voronoi formula \eqref{eq:glnVoronoi} is proved to be equivalent to a formula \eqref{eq:HequalsG} involving Gauss sums.
The equivalence is shown in Proposition \ref{equivalent}. 
A convoluted version of \eqref{eq:HequalsG} is obtained in Theorem \ref{gln-average-average-voronoi} by comparing Dirichlet coefficients
of two different expressions of a double Dirichlet series. We later show in Proposition \ref{mobius} that this convoluted version yields \eqref{eq:HequalsG}.
\end{remark}

\begin{remark}\label{eulerian}
If we start with an $L$-series $L(s,F)$ with an Euler product 
$$L(s,F)=\sum_{n=1}^\infty \frac{A(1,\dots,1,n)}{n^s}=\prod_p \prod_{i=1}^N\bra{1-\frac{\alpha_i(p)}{p^s}}\inv$$
and with $\prod_i \alpha_i(p)=1$ for any $p$,
we can define $A(p^{k_1},\dots,p^{k_{N-1}})$ by the Casselman-Shalika formula (Proposition 5.1 of \cite{zhou2}) and they are compatible with the Hecke relations.
More explicitly, for a prime number $p$, 
we define $A(p^{k_1},\dots,p^{k_{N-1}}) = S_{k_1,\dots,k_{N-1}}(\alpha_1(p),\dots, \alpha_N(p))$ by the work of Shintani
where $S_{k_1,\dots,k_{N-1}}(x_1,\dots,x_N)$ is the Schur polynomial, which can be found 
in \cite[p. 233]{goldfeld}.

We extend the definition to all $A(*,\dots,*)$ multiplicatively by \eqref{glnheckerelation0}.
One can prove that $A(*,\ldots,*)$ satisfies the Hecke relations \eqref{glnheckerelation0}, \eqref{glnheckerelation1}, \eqref{glnheckerelation2}.
In summary, the ``coefficients'' $A(*,\ldots,*)$ along with the Hecke relations can be generated by an $L$-function with an Euler product.
\end{remark}




The following examples satisfy the conditions in Theorem \ref{conversethm} and hence we have a Voronoi formula for each of them.
	
\begin{example}[{Automorphic form for $\ssl{N}{Z}$}]  	
Any cuspidal automorphic form for $\ssl{N}{Z}$
satisfies the conditions in Theorem \ref{conversethm}.
It can have an unramified or ramified component at the archimedean place, because only the exact form of the $G_\pm$ function would change (see \cite{godementjacquet}).
The Hecke-Maass cusp forms considered in Theorem \ref{thm:glnVoronoi} are included in this category, and therefore, we   prove Theorem \ref{conversethm} instead of Theorem \ref{thm:glnVoronoi}. 		
\end{example}


\begin{example}[{Rankin-Selberg convolution}] \label{ex:rankinSelberg}
Let $F_1$ and $F_2$ be even Hecke-Maass cusp forms for $\ssl{N_1}{Z}$ and $\ssl{N_2}{Z}$ 
with the spectral parameters $(\lambda_1,\dots,\lambda_{N_1})\in \mathbb{C}^{N_1}$ and $(\mu_1,\dots,\mu_{N_2})\in \mathbb{C}^{N_2}$ respectively.
Assume $F_1\neq \widetilde F_2$ if $N_1=N_2$.
The automorphic forms $F_1$ and $F_2$ have the standard $L$-functions
$$L(s,F_1)=\prod_p \prod_{i=1}^{N_1}\bra{1-\frac{\alpha_i(p)}{p^s}}\inv 
\qquad \text{ and } \qquad 
L(s,F_2) = \prod_p \prod_{i=1}^{N_2}\bra{1-\frac{\beta_i(p)}{p^s}}\inv.$$

Let $L(s,F_1\times F_2)$ be the Rankin-Selberg $L$-function of $F_1$ and $F_2$ defined by $$L(s, F_1\times F_2)=\prod_p \prod_{i_1=1}^{N_1}\prod_{i_2=1}^{N_2}\bra{1-\frac{\alpha_{i_1}(p)\beta_{i_2}(p)}{p^s}}\inv .$$
The $L$-function is of degree $N:=N_1N_2$.
The work of Jacquet, Piatetski-Shapiro, and Shalika \cite{JPSS} shows that 
$L(s,F\times \chis)=L(s, (F_1\times \chis)\times F_2)$ is holomorphic and satisfies the functional equation \eqref{gln-FE-character} for $F:=F_1\times F_2$.


Define $A(p^{k_1}, \dots, p^{k_{N-1}})$
by the Schur polynomials as  in Remark \ref{eulerian} 
$$A(p^{k_1}, \dots, p^{k_{N-1}}):=S_{k_1,\dots, k_{N-1}}  \bra{\alpha_1(p)\beta_1(p),\dots, \alpha_{i_1}(p)\beta_{i_2}(p),\dots, \alpha_{N_1}(p)\beta_{N_2}(p)}.$$
Extend the definition to all $A( *,\dots*)$ multiplicatively by \eqref{glnheckerelation0}. 
Define 	$$G_\pm(s):= {i^{-N\delta}}\pi^{-N(1/2-s)}\prod_{i_1=1}^{N_1}\prod_{i_2=1}^{N_2} \Gamma\bra{\frac{\delta + 1-s-\overline{\lambda_{i_1}}-\overline{\mu_{i_2}}}{2}}\Gamma\bra{\frac{\delta + s-\lambda_{i_1}-\mu_{i_2}}{2}}^{-1},$$
where one takes $\delta = 0$ or $\delta = 1$ for $G_+$and $G_-$ respectively.
 Theorem \ref{conversethm} gives us a Voronoi formula for the Rankin-Selberg convolution $F=F_1\times F_2$ with the $A( *,\dots* )$ and $G_\pm$ defined above.
\end{example}

\begin{example}[{Isobaric sum, Eisenstein series}]\label{ex:isobaric}
For $i=1,\dots,k$ let $F_i$ be a Hecke-Maass cusp form for $\ssl{N_i}{Z}$. 
Let $s_i$ be complex numbers with $\sum_i N_i s_i=0$.
Define the isobaric sum $F=\bra{F_1\times  |\cdot|^{s_1}_\mathbb{A} } \boxplus\bra{ F_2 \times  |\cdot|^{s_2}_\mathbb{A}}\boxplus \cdots \boxplus\bra{ F_k\times  |\cdot|^{s_k}_\mathbb{A}}$, whose $L$-function is
$L(s,F)=\prod_i L(s+s_i  ,F_i)$.  
This isobaric sum $F$ is associated with a non-cuspidal automorphic form on GL($N$), an Eisenstein series twisted by Maass forms,  where $N=\sum_i N_i$ (see \cite[Section 10.5]{goldfeld}). The $L$-function twisted by a character is simply given by $L(s,F\times \chi^*) = \prod_{i}L(s + s_i, F_i\times\chi^*)$ which satisfies the conditions of Theorem \ref{conversethm}.
\end{example}

\begin{example}[{Symmetric powers on GL(2)}]
Let $f$ be a modular form of weight $k$ for $\ssl{2}{Z}$ and define $F:=\textrm{Sym}^2 f$. The symmetric square $F$ satisfies the conditions in Theorem \ref{conversethm} by the work of Shimura, \cite{shimura}. Here we do not need to involve automorphy using Gelbart-Jacquet lifting. 
One may have similar results for higher symmetric powers depending on the recent progress in the theory of Galois representations.
\end{example}


As a last remark, let us explain the construction of the double Dirichlet series $Z(s,w)$ given by \eqref{eq:zsw}. This construction 
originates from the Rankin-Selberg convolution of a cusp form $F$ and an Eisenstein series on $\GL(2)$. 
The Fourier coefficients of the Eisenstein series $E(z,s,\chis)$ can be written in terms of the divisor function $\sigma_{2s-1}(n,\chis)$ defined in  Definition \ref{definegausssum}:
\[
	\frac{1}{n^{2s-1}}\frac{\sigma_{2s - 1}(n,\chi^*)}{L(2s,\overline{\chi^*})} 
	\quad \text{ or } \quad \sum_{\ell=1}^\infty \frac{g(\overline{\chi^*}, \ell c^*, n)}{(\ell c^*)^{2s}}.
\]
Therefore, in the case of $F$ on $\GL(2)$, the Rankin-Selberg integral of $F$ and $E(*,w-s+1/2,\chis) $ produces the double Dirichlet series
$$ \sum_{n=1}^\infty \sum_{\ell=1}^\infty \frac{A(n) g(\overline{\chi^*}, \ell c^*,n)}{n^s (\ell c^*)^{2w + 1-2s}}.$$
A similar expression appears on the left hand side of the Voronoi formula with Gauss sums \eqref{eq:Hdefn}.
The Rankin-Selberg convolution of the cusp form $F$ and an Eisenstein series can be written as a product of two copies of standard $L$-function of $F$, namely
$$ 
	\frac{L(2w-s,F)L(s,F\times\chi^*)}{L(2w-2s + 1,\overline{\chi^*})}.$$
Applying the functional equation to only $L(s,F\times \chis)$ gives us another expression, which is similar to the right hand side \eqref{eq:Gdefn} of the Voronoi formula with Gauss sums.
Since $L(2w-s, F)$ was not used in this process, we have the freedom to replace $L(2w-s,F)$ by $L_\vec{q}(2w-s,F)$ in the case of $\GL(N)$ and it gives us enough generality to prove the Voronoi formula \eqref{eq:HequalsG} with Gauss sums.
In the case of $\GL(3)$, this construction is similar to  Bump's double Dirichlet series, see \cite[Chapter 6.6]{goldfeld} or \cite[Chapter X]{bump}.

\section{Background on Gauss sums}\label{sec:background}

Here we collect information about the Gauss sums of Dirichlet characters which are not necessarily primitive.

\begin{definition}\label{definegausssum}
Let $\chi$ be a Dirichlet character modulo $c$ induced from a primitive Dirichlet character $\chis$ modulo $c^*$. 
Define the divisor function $$\sigma_s(m,\chi)=\sum_{d|m}\chi(d)d^{s}.$$
Define the Gauss sum of $\chi$
$$g(\chis, c, m)=\sum_{\underset{\scriptstyle (u, c)=1}{u\mod c}}\chi(u)e\bra{\frac{mu}{c}},$$
and the standard Gauss sum for $\chis$ is given as $\tau(\chis)=g(\chis,c^*,1)$.

The Gauss sum $g(\chis,c,m)$ is the same as the Gauss sum $\tau_m(\chi)$ in other literature. However we prefer our notation because we come upon numerous Gauss sums of characters $\chi$ induced from a single primitive character $\chis$.

\end{definition}

\begin{lemma}[Gauss sum of non-primitive characters, Lemma 3.1.3.(2) of \cite{miyake}]\label{inducedgauss}
Let $\chi$ be a character modulo $c$ induced from primitive character $\chi^*$ modulo $c^*$.
Then the Gauss sum of $\chi$ is given by
$$g(\chis, c, a)
=\tau(\chi^*)\sum_{d|\bra{a,\frac{c}{c^*}}}d\chi^*\bra{\frac{c}{c^*d}}\overline{\chi^*}\bra{\frac{a}{d}}\mu\bra{\frac{c}{c^*d}}.$$
\end{lemma}


\begin{lemma}[Theorem 9.12 of \cite{Montgomery}]\label{inducedgauss2}
Let $\chi^*$ be a primitive character modulo $c^*$ and assume $c^*|c$. Then, we have
\[
	g(\chi^*,c,a)= \tau(\chi^*)\frac{\phi(c)}{\phi\bra{\frac{c}{(c,a)}}}\mu\bra{\frac{c}{c^*(c,a)}}\chis\bra{\frac{c}{c^*(c,a)}}\chibs\bra{\frac{a}{(c,a)}},
\]
if $c^*|c/(a,c)$. Otherwise $g(\chi^*,c,a)$ is zero.
\end{lemma}



Next lemma is a generalization of a famous formula of Ramanujan,
$$\frac{\sigma_{s-1}(n)}{n^{s-1}}=\zeta(s)\sum_{\ell =1}^\infty \frac{c_\ell(n)}{\ell^s},$$
where $c_\ell(n)$ is the Ramanujan sum.

\begin{lemma}\label{divisorcharacter}
Let $\Re(s)>1$. Define a Dirichlet series  $$I(s,\chi^*,c^*,m)=\sum_{\ell=1}^\infty \frac{g(\chis,\ell c^*,m)}{\ell^s},$$
as a generating function for the non-primitive Gauss sums induced from $\chi^*$. It satisfies the identity
 $$\tau(\chi^*)\sigma_{s-1}(m,\overline{\chi^*})=m^{s-1}I(s,\chi^*,c^*,m)L(s,\chi^*).$$
\end{lemma}

\begin{proof}
We prove the equivalent formula $\tau(\chi^*)m^{1-s}\sigma_{s-1}(m,\overline{\chi^*})L(s,\chi^*)^{-1}=I(s,\chi^*,c^*,m)$.
Expand its both sides  
and the left hand side is $$\tau(\chis) \sum_{d|m}\frac{d\chibs(m/d)}{d^s}\sum_{n=1}^\infty \frac{\chis(n)\mu(n)}{n^s} = 
\tau(\chis) \sum_{\ell=1}^\infty \frac{\sum_{d|(m,\ell)} d\chibs(m/d)\mu(\ell/d)\chis(\ell/d)}{\ell^s},$$
which equals the right side by Lemma \ref{inducedgauss}.
\end{proof}

\begin{lemma}\label{lem:averageGauss}
	For any two positive integers $n$ and $m$, and a primitive Dirichlet character $\chis$ modulo $c^*$, we have
	\[
		\sum_{\substack{\ell d = n}} \chi^*(d) g(\chi^*,\ell c^*, m) =
		\begin{cases}
		\tau(\chi^*)\overline{\chi^*(m/n)} n, &\text{if } n|m,\\
		0 ,&\text{otherwise.} 
		\end{cases}
	\]
\end{lemma}
\begin{proof}
	We start with the formula,
	\[
		\frac{\tau(\chi^*) \sigma_{s-1}(m,\overline{\chi^*})}{m^{s-1}} = I(s,\chi^*,c^*,m) L(s,\chi^*).
	\]
	Both sides are Dirichlet series and we equate coefficients. The left hand side is given as
	\[
		\tau(\chi^*) \sum_{e|m} \frac{\overline{\chi^*(m/e)}e}{e^s},
	\]
	whereas the right hand side is
	\[
		\sum_{\ell=1}^\infty \frac{g(\chi^*,\ell c^*,m)}{\ell^s} 
		\sum_{d = 1}^\infty \frac{\chi^*(d)}{d^s} 
		= \sum_{n=1}^\infty \frac{\sum_{d\ell = n}\chi^*(d) g(\chi^*,\ell c^*, m)}{n^s}.\qedhere
	\]
\end{proof}


\section{The Voronoi formula}\label{sec:gln}

\subsection{Double Dirichlet series}

We begin by proving a convoluted version of \eqref{eq:HequalsG}. 

\begin{theorem}\label{gln-average-average-voronoi}
For $N\geq 3$, $\vec{q}=(q_1,\dots,q_{N-2})\in \mathbb{N}^{N-2}$, and $n\in \mathbb{N}$, define
$$\hhh(\vec{q};n,s)  :=   \sum_{d_1| q_1, \ldots, d_{N-2}|q_{N-2}}
		\frac{\chi^*(d_1\cdots d_{N-2})}{(d_1 \cdots d_{N-2})^s}  \sum_{d\ell = n} \chi^*(d) H(\vec{q}'; \ell c^*, \chi^*,s) \quad  \text{ for }\Re(s)\gg 1  $$
and 
$$\ggg(\vec{q};n,s)  :=   \sum_{d_1| q_1, \ldots, d_{N-2}|q_{N-2}}
		\frac{\chi^*(d_1\cdots d_{N-2})}{(d_1 \cdots d_{N-2})^s}  \sum_{d\ell = n} \chi^*(d) G(\vec{q}'; \ell c^*, \chi^*,s)\quad \text{ for }\Re(1-s)\gg 1  ,$$
where we denote for abbreviation
\begin{equation}\label{eq:qprime}
\vec{q}' = (\tfrac{q_1d}{d_1}, \tfrac{q_2d_1}{d_2},\ldots,\tfrac{q_{N-2}d_{N-3}}{d_{N-2}}).
\end{equation}
The functions $\hhh(\vec{q};n,s)$ and $\ggg(\vec{q};n,s)$ have 
analytic continuation to all $s\in \C$ and these analytic continuations satisfy
\begin{equation}\label{eq:HGside}
\hhh(\vec{q};n,s) = \ggg(\vec{q};n,s).
\end{equation}
\end{theorem}



\begin{proof}
The region of absolute convergence for $\hhh(\vec{q};n,s)$ is a right half plane $\Re(s)\gg 1$, and the region of absolute convergence of $\ggg(\vec{q};n;s)$ is a left half plane $\Re(1-s) \gg1$.
Let $Z(s,w)$ be defined as in \eqref{eq:zsw}.
For any $s \in \C$ and $w$ with $\Re(w)$ large enough so that $\Re(2w-s)\gg 1 $ and $\Re(w-s)> 0$, 
writing $L_{\vec{q}}(2w-s,F)$ and $L(2w-2s+1,\chibs)\inv$
as Dirichlet series, we derive 
\begin{equation*}
Z(s,w) = L(s,F\times\chis)\sum_{n=1}^\infty
 \frac{\sum_{d|n}A(q_{N-2},\dots,q_1,d)d^s\chibs(n/d)\mu(n/d)(n/d)^{2s-1}}{n^{2w}}.
\end{equation*}
Hence, we have 
$$	Z(s,w) = \sum_{n=1}^\infty \frac{a_n(s)}{n^{2w}}.
$$
where 
$$a_n(s)=L(s,F\times \chis)\sum_{d|n}A(q_{N-2},\dots,q_1,d)d^s\chibs(n/d)\mu(n/d)(n/d)^{2s-1}.$$
Here $a_n(s)$ is an analytic function of $s\in \C$, because $L(s,F\times\chi^*)$ is entire. 
The computation below shows that $a_n(s)$ equals either side of \eqref{eq:HGside} in their respective regions of absolute convergence, up to scaling by a constant $\tau(\chibs)$. This proves the analytic continuation of $\hhh$ and $\ggg$ as well as their equality. 

For $\Re(s)\gg 1,\Re(w-s) >0 $, we expand the two $L$-functions in the numerator of $Z(s,w)$ as Dirichlet series, obtaining
	\begin{align*}
		Z(s,w)&=\frac{1}{L(2w-2s + 1,\overline{\chi^*})} \sum_{n,m=1}^\infty 
		\frac{A(q_{N-2}, \ldots, q_1, n) A(1,\ldots, 1, m) \chi^*(m)}{n^{2w-s}m^s}\\
		&=\frac{1}{L(2w-2s + 1,\overline{\chi^*})} \sum_{n,m=1}^\infty 
		\frac{\chi^*(m)}{n^{2w-s}m^s} 
		\mspace{-20mu}
		\sum_{\substack{d_0 d_1 \cdots d_{N-1}= m\\ d_0|n, d_1| q_1, \ldots, d_{N-2}|q_{N-2}}}
		\mspace{-25mu}
		A\left(\tfrac{q_{N-2}d_{N-3}}{d_{N-2}}, \ldots, \tfrac{q_1d_0}{d_1}, \tfrac{nd_{N-1}}{d_0}\right),
		\end{align*}
where we have used the Hecke relation \eqref{glnheckerelation1}.
We change the variable $n/d_0\to n$ and combine $h = nd_{N-1}$, giving
\begin{align*}
Z(s,w)&=\frac{1}{L(2w-2s + 1,\overline{\chi^*})} \sum_{n,d_0,d_{N-1}=1}^\infty 
\sum_{ \substack{d_i| q_i\\ i=1,\ldots, N-2}}
\frac{\chi^*(d_0\dots d_{N-1})}{n^{2w-s}d_0^{2w-s}(d_0\dots d_{N-1})^s} \\
&\cquad \times A\left(\tfrac{q_{N-2}d_{N-3}}{d_{N-2}}, \ldots, \tfrac{q_1d_0}{d_1}, nd_{N-1}\right)\\
&=		\frac{1}{L(2w-2s + 1,\overline{\chi^*})} \sum_{d_0,\nn =1}^\infty 
\sum_{ \substack{d_i| q_i\\i = 1,\ldots, N-2}}
\frac{\chi^*(d_0\dots d_{N-2})}{d_0^{2w-s}(d_0\dots d_{N-2})^s} \\
&\cquad\times A\left(\tfrac{q_{N-2}d_{N-3}}{d_{N-2}}, \ldots, \tfrac{q_1d_0}{d_1}, \nn\right)\frac{\sigma_{2w-2s}(\nn,\chis)}{\nn^{2w-s}}.
\end{align*}
Applying Lemma \ref{divisorcharacter}, we get 
\begin{align*}
Z(s,w)=& \tau(\chibs)^{-1} \sum_{d_0=1}^\infty 
\sum_{ \substack{d_i| q_i\\ i=1,\ldots, N-2}}
\frac{\chi^*(d_0\dots d_{N-2})}{d_0^{2w}(d_1\dots d_{N-2})^s} 
\sum_{\nn=1}^\infty
 \frac{A\left(\tfrac{q_{N-2}d_{N-3}}{d_{N-2}}, \ldots, \tfrac{q_1d_0}{d_1}, \nn\right)}{\nn^s}
\sum_{\ell=1}^\infty \frac{g(\chibs, \ell c^*,\nn) }{\ell^{2w-2s+1}}.
\end{align*}
Therefore we reach
	\begin{equation}\label{sumofH}
		Z(s,w) = \tau(\chibs)^{-1}  \sum_{n=1}^\infty \frac{1}{n^{2w}}\sum_{  d_1| q_1, \ldots, d_{N-2}|q_{N-2}}
		\frac{\chi^*(d_1\cdots d_{N-2})}{(d_1 \cdots d_{N-2})^s}  \sum_{d\ell = n} \chi^*(d) H(\vec{q}'; \ell c^*, \chi^*,s),
	\end{equation}
	where $\vec{q}'$ is defined in \eqref{eq:qprime}.
	\\


On the other hand, let us apply the functional equation \eqref{gln-FE-character} to $L(s,F\times \chis)$ in $Z(s,w)$, giving  
$$Z(s,w)=\frac{G(s)\tau(\chis)^{N}}{{c^*}^{Ns}}
\frac{L_{\vec{q}}(2w-s, F) L(1-s,\widetilde F\times\chibs)}{L(2w-2s + 1,\overline{\chi^*})}.$$
Given $\Re(1-s)\gg 1$ and $\Re(2w-s)\gg 1$, 
we open the expression as a Dirichlet series,
\begin{align*}
Z(s,w)=&\frac{G(s)\tau(\chis)^N{c^*}^{-Ns}}{L(2w-2s + 1,\overline{\chi^*})} \sum_{n,m=1}^\infty 
		\frac{A(q_{N-2}, \ldots, q_1, n) A(m,1,\dots,1) \chibs(m)}{n^{2w-s}m^{1-s}}
\\
=&\frac{G(s)\tau(\chis)^N{c^*}^{-Ns}}{L(2w-2s + 1,\overline{\chi^*})} \sum_{n,m=1}^\infty 
		\frac{\chibs(m)}{n^{2w-s}m^{1-s}} 
		\sum_{\substack{d_0 d_1 \cdots d_{N-1}= m\\ d_0|n, d_1| q_1, \ldots, d_{N-2}|q_{N-2}}}
		A\left(\tfrac{q_{N-2}d_{N-1}}{d_{N-2}}, \ldots, \tfrac{q_1d_2}{d_1}, \tfrac{nd_{1}}{d_0}\right)\\
		=&\frac{G(s)\tau(\chis)^N{c^*}^{-Ns}}{L(2w-2s + 1,\overline{\chi^*})} \sum_{n,m=1}^\infty 
		\sum_{\substack{d_0 d_1 \cdots d_{N-1}= m\\ d_0|n, d_1| q_1, \ldots, d_{N-2}|q_{N-2}}}
		\mspace{-35mu}
		\frac{\chibs(d_0d_1\dots d_{N-1})A\left(\tfrac{q_{N-2}d_{N-1}}{d_{N-2}}, \ldots, \tfrac{q_1d_2}{d_1}, \tfrac{nd_{1}}{d_0}\right)}{(n/d_0)^{2w-s}d_0^{1+2w-2s}(d_1\dots d_{N-1})^{1-s}}, 
\end{align*}		
where we have combined the Fourier coefficients by the Hecke relation \eqref{glnheckerelation2}.We change the variable $n/d_0\to n$. Then the sum over $d_0$ cancels with $L(2w-2s+1,\chibs)$ in the denominator, giving
\begin{align}
Z(s,w) &=\frac{G(s)\tau(\chis)^N{c^*}^{-Ns}}{L(2w-2s + 1,\overline{\chi^*})} \sum_{n,d_0,d_{N-1}=1}^\infty 
			\sum_{\substack{d_i| q_i\\ i = 1, \ldots, N-2}}
			\mspace{-15mu}
		\frac{\chibs(d_0d_1\dots d_{N-1})A\left(\tfrac{q_{N-2}d_{N-1}}{d_{N-2}}, \ldots, \tfrac{q_1d_2}{d_1}, d_{1}n\right)}{n^{2w-s}d_0^{1+2w-2s}(d_1\dots d_{N-1})^{1-s}}
		\notag \\ 
	 &= \frac{G(s)\tau(\chis)^N}{{c^*}^{Ns}} \sum_{n, d_{N-1}=1}^\infty 
			\sum_{\substack{d_i| q_i\\ i =1,\ldots, N-2}}
		\frac{\chibs(d_1\dots d_{N-1})A\left(\tfrac{q_{N-2}d_{N-1}}{d_{N-2}}, \ldots, \tfrac{q_1d_2}{d_1}, d_{1}n\right)}{n^{2w-s}(d_1\dots d_{N-1})^{1-s}} . \label{sumofG2}
\end{align}

If we denote the right hand side of \eqref{eq:HGside} by $\tau(\chibs) b_n(s)$, our goal is to transform \eqref{sumofG2} into $R:= \sum_{n=1}^\infty b_n(s)n^{-2w}$. But at this point it is easier to start from $R$. More explicitly, we have
\begin{equation}\label{sumofG}
R=\tau(\chibs)^{-1} \sum_{\nn=1}^\infty \frac{1}{\nn^{2w}}\sum_{  d_1| q_1, \ldots, d_{N-2}|q_{N-2}}
		\frac{\chi^*(d_1\cdots d_{N-2})}{(d_1 \cdots d_{N-2})^s}  \sum_{d\ell = \nn} \chi^*(d) G(\vec{q}'; \ell c^*, \chi^*,s).\end{equation}
Here $\vec{q}'$ has been defined in \eqref{eq:qprime}.  
We plug in the definition of $G(\vec{q}';\ell c^*,\chis,s)$ from \eqref{eq:Gdefn} for $\vec{q}'$, giving
%
\begin{align*}
		G&(\vec{q}'; \ell c^*, \chi^*,s) 
		=\frac{G(s) \chis(-1)  }{{{c^*}^{Ns-1}\ell^{(N-2)s}}}
\sum_{f_1|\frac{q_1d\ell }{d_1}}\sum_{f_2 | \frac{q_1q_2d\ell }{f_1d_2}}   
		\cdots \sum_{f_{N-2}| \frac{q_1\ldots q_{N-2} d\ell }{f_1 \ldots f_{N-3}d_{N-2}}}		
\\
		&\times \sum_{n=1}^\infty 
		\frac{A(n,f_{N-2}, \ldots, f_1)}
		{n^{1-s}f_1 f_2 \cdots f_{N-2}}
		\frac{f_1^{(N-1)s}f_2^{(N-2)s}\cdots f_{N-2}^{2s}}
		{q_1^{(N-2)s}q_2^{(N-3)s}\cdots q_{N-2}^{s}}\frac{(d_1\dots d_{N-2})^{s}}{d^{(N-2)s}}\\
		&\times g(\chi^*, \ell c^*,f_1) g(\chi^*, \tfrac{q_1d\ell c^*}{f_1d_1},f_2 ) 
		\cdots g(\chi^*, \tfrac{q_1\cdots q_{N-3}d\ell c^*}{f_1 \cdots f_{N-3}d_{N-3}}, f_{N-2})
		g(\chi^*, \tfrac{q_1\cdots q_{N-2}d\ell c^*}{f_1\cdots f_{N-2}d_{N-2}},n).
	\end{align*}
We substitute $G(\vec{q}'; \ell c^*, \chi^*,s) $ with this expression in \eqref{sumofG} and change the orders of summation between $f_i$ and $d_i$. The summations over $d$ and $d_i$ collapse with the repeated use of Lemma \ref{lem:averageGauss}, giving
	\begin{align*}
		R=&\; \tau(\chibs)^{-1} \frac{G(s) \chis(-1)}{{c^*}^{Ns-1}}\sum_{\nn =1}^\infty 
		\sum_{n=1}^\infty	
		\sum_{\substack{\nn|f_1\\ f_1 | q_1\nn}}
		\sum_{\substack{\frac{ q_1 \nn}{f_1}|f_2\\ f_2| \frac{q_1q_2\nn}{f_1}}}\cdots
		\sum_{\substack{\frac{ q_1\cdots q_{N-3}\nn}{f_1\cdots f_{N-3}}|f_{N-2}\\f_{N-2} |\frac{ q_1 \cdots q_{N-2} \nn}{f_1 \cdots f_{N-3} } }}
\sum_{\frac{ q_1\dots q_{N-2} \nn }{f_1\cdots f_{N-2}}|n}
\\
		& \frac{\tau(\chis)^{N-1}}{\nn^{2w}} \chibs\bra{\frac{f_1}{\nn}}
		    \chibs\bra{\frac{f_1f_2}{\nn q_1}}
		    \dots
		    \chibs\bra{\frac{f_1f_2\cdots f_{N-2}}{\nn q_1\cdots q_{N-3}}}
		    \chibs\bra{\frac{f_1f_2\cdots f_{N-2}n}{\nn q_1\cdots q_{N-2}}}
		    \\
		    &\times  \bra{\frac{q_1}{f_1}}^{N-2}
		    \bra{\frac{q_2}{f_2}}^{N-3}
		    \cdots\bra{\frac{q_{N-2}}{f_{N-2}}}\nn^{N-1-Ns+2s}
\frac{A(n,f_{N-2},\ldots,f_1)}{n^{1-s}f_1\cdots f_{N-2}}\frac{f_1^{(N-1)s}\cdots f_{N-2}^{2s}}{q_1^{(N-2)s}\cdots q_{N-2}^s}.
\end{align*}
Define the variables $e_1=f_1/\nn$ and 
$e_i =f_1\cdots f_i/q_1\cdots q_{i-1}\nn$ for $i = 2, \ldots N-2$. 
The double conditions under the sums simplify to $e_i | q_i$. Also define $e_{N-1} = f_1\cdots f_{N-2}n/ \nn q_1\cdots q_{N-2}$ 
and it runs over all positive integers. Finally noting 
$\tau(\chibs) \inv =\chis(-1) \tau(\chis)/c^*$ , we get
\begin{equation*}	
	R	=
	\frac{G(s)\tau(\chis)^{N}}{{c^*}^{Ns}}
	\sum_{\nn ,e_{N-1}=1}^\infty \frac{1}{\nn^{2w-s}}  
	\sum_{\substack{e_i|q_i\\ i = 1,\ldots, N-2}}\frac{\chibs(e_1\cdots e_{N-2}e_{N-1})}{(e_1\cdots e_{N-1})^{1-s}}  A(\tfrac{e_{N-1}q_{N-2}}{e_{N-2}}, \ldots,\tfrac{e_2q_1}{e_1}, e_1h),
	\end{equation*}		
which in turn, by \eqref{sumofG2}, equals $Z(s,w)$ as well as  \eqref{sumofH}.  We complete the proof after applying the uniqueness theorem for Dirichlet series (Theorem 11.3 of \cite{apostol})
to the equality between \eqref{sumofH} and \eqref{sumofG}.
\end{proof}

\begin{remark}\label{remark-GL2}
The above proof works for $N\geq 3$ but not for $N=2$. We can prove the Voronoi formula for $\ssl{2}{Z}$ similarly and easily by considering 
$$Z(s,w)=\frac{L(2w-s,F)L(s,F\times\chis)}{L(2w-2s+1,\chibs) L(2w,\chis )}.$$ 
We have from the Hecke relations on GL(2)
$$Z(s,w)=\tau(\chibs)\inv \sum_{\ell=1}^\infty  \sum_{n=1}^\infty \frac{A(n)}{n^s}\frac{g(\chibs,\ell c^*,n)}{\ell^{1+2w-2s}},$$
and applying the functional equation for $L(s,F\times\chis)$ we have
$$Z(s,w)=\tau(\chis){c^*}^{-2s}G(s) \sum_{\ell=1}^\infty \sum_{n=1}^\infty
\frac{A(n)}{n^{1-s}}\frac{g(\chis,\ell c^*,n)}{\ell^{2w}}.$$
Applying the uniqueness theorem for Dirichlet series 
to the variable $w$, we get the Voronoi formula with Gauss sums on $\GL(2)$.
\end{remark}


\begin{proposition}\label{mobius}
Equation \eqref{eq:HequalsG} is equivalent to Theorem \ref{gln-average-average-voronoi}.
\end{proposition}
\begin{proof}
Construct the following summation
\begin{eqnarray*}
\TT  &:=&\sum_{e_0|n} \sum_{e_1| q_1e_0}\cdots \sum_{e_{N-2}|q_{N-2}e_{N-3}}
\frac{\mu(e_0\cdots e_{N-2})\chis(e_0\cdots e_{N-2})}{(e_1\cdots e_{N-2})^s}  \hhh  \bra{\frac{q_1e_0}{e_1},\dots,\frac{q_{N-2}e_{N-3}}{e_{N-2}} ; \frac{n}{e_0}  ,s}\\
&=& \sum_{e_0|n} \sum_{e_1| q_1e_0}\cdots \sum_{e_{N-2}|q_{N-2}e_{N-3}}
\frac{\mu(e_0\cdots e_{N-2})\chis(e_0\cdots e_{N-2})}{(e_1\cdots e_{N-2})^s}   
 \sum_{\substack{ d_i| q_i e_{i-1}/e_i\\i=1,\dots,N-2}}
		\frac{\chi^*(d_1\cdots d_{N-2})}{(d_1 \cdots d_{N-2})^s}   \\
&&	\aquad \times		\sum_{d_0| n/e_0}  
		\chi^*(d_0) H\bra{ \frac{q_1 e_0d_0 }{e_1d_1},\dots, \frac{q_{N-2}e_{N-3}d_{N-3}}{e_{N-2}d_{N-2}}  ; \frac{n}{e_0d_0} c^*   ,\chis,s        }   .
\end{eqnarray*}
Change variables $e_id_i\to a_i$ for $i=0,\dots,N-2$, and change orders of summation, getting
\begin{eqnarray*}
	\TT &=& \sum_{a_0|n}\sum_{e_0|a_0} \sum_{a_1 | q_1e_0} \sum_{e_1|a_1} \cdots \sum_{a_{N-2}|q_{N-2}e_{N-3} }\sum_{e_{N-2}|a_{N-2}} \frac{\chi^*(a_0\cdots a_{N-2})}{(a_1\cdots a_{N-2})^s}\\
	&& \bquad\times H\left(\frac{q_1a_0}{a_1}, \frac{q_2a_1}{a_2}, \ldots, \frac{q_{N-2}a_{N-3}}{a_{N-2}}; \frac{nc^*}{a_0},\chi^*,s\right) \mu(e_0) \cdots\mu(e_{N-2}).
\end{eqnarray*}		
One by one the M\"obius summation over $e_i$ will force $a_i = 1$, and we obtain $\TT = H(\vec{q}; nc^*, \chi^*, s).$ By Theorem \ref{gln-average-average-voronoi}, we have $\hhh = \ggg$ and and the same calculations yield $T=G(\vec{q}; nc^*, \chis, s)$. This proves the theorem.
\end{proof}

\subsection{Equivalence between equations \eqref{eq:glnVoronoi} and \eqref{eq:HequalsG}}\label{subsec:equivalent}

First we prove a lemma showing that the hyper-Kloosterman sum on the right hand side of \eqref{eq:glnVoronoi} becomes a product of $(N-2)$ Gauss sums after averaging against a Dirichlet character.
\begin{lemma}\label{lem:averageGLnKloosterman}
	Let $\chi$ be a Dirichlet character modulo $c$ which is induced from the primitive 
	character $\chi^*$ modulo $c^*$.
	Let $\vec{q} = (q_1, \ldots, q_{N-2})$ and $\vec{d} = (d_1,\ldots,d_{N-2})$ be 
	two tuples of positive integers, and assume that all the divisibility conditions in 
	\eqref{eq:divisibilityConditions} are met.
Consider the summation
	\[
		S := \sum_{\substack{a \mod c\\(a,c)=1}} \chi(a) \Kl(a,n,c;\vec{q},\vec{d}).
	\]
	The quantity $S$ is zero unless the divisibility conditions
\begin{equation}	\label{divisibilitycondition2}	d_1c^*|q_1c, \quad d_2c^*\left| \frac{q_1q_2c}{d_1},\right. \quad d_3c^*\left| \frac{q_1q_2q_3c}{d_1d_2},
		\quad \ldots, \quad d_{N-2}c^*\right|\frac{q_1\cdots q_{N-2}c}{d_1\cdots d_{N-3}},
\end{equation}	
are satisfied.
Under such divisibility conditions, $S$ can be written as a product of Gauss sums
	\[
		S = g(\chi^*, c,d_1) g(\chi^*, \tfrac{q_1c}{d_1},d_2) 
		\cdots g(\chi^*, \tfrac{q_1\cdots q_{N-3}c}{d_1 \cdots d_{N-3}}, d_{N-2})
		g(\chi^*, \tfrac{q_1\cdots q_{N-2}c}{d_1\cdots d_{N-2}},n).
	\]
\end{lemma}


\begin{proof}

The divisibility conditions \eqref{eq:divisibilityConditions} imply
\begin{equation}\label{divisibilitycondition3}
d_1|q_1(c,d_1),\quad  d_2\left|q_2\left(\frac{q_1c}{d_1},d_2\right),\quad d_3\right|\left. q_3\bra{\frac{q_1q_2c}{d_1d_2},d_3},\dots,
d_{N-2}\right|q_{N-2}\bra{\frac{q_1\dots q_{N-3}c}{d_1\dots d_{N-3}},d_{N-2}}.
\end{equation}
We  open up the hyper-Kloosterman sum in $S$.
Our forthcoming computation is an iterative process. The summation over $a$ yields a Gauss sum, which in turn produces  the term $\chibs(x_1)$. Then the summation over $x_1$ yields another Gauss sum, which produces the term $\chibs(x_2)$ and so on. 

Firstly we sum over $a$ modulo $c$
	\begin{align*}
		S =& \sum_{a \mod c} \chi(a) \sums_{x_1 (\text{mod }{\frac{q_1c}{d_1}})} 
		e\left(\frac{d_1x_1a}{c} \right) \left(\sums_{x_2 (\text{mod }{\frac{q_1q_2c}{d_1d_2}})}
		e\left(\frac{d_2x_2 \overline{x_1}}{\frac{q_1c}{d_1}}\right) \cdots\right)\\
		=& \sums_{x_1 (\text{mod }{\frac{q_1c}{d_1}})} g(\chi^*,c, x_1d_1)
		\left(\sums_{x_2 (\text{mod }{\frac{q_1q_2c}{d_1d_2}})}
		e\left(\frac{d_2x_2 \overline{x_1}}{\frac{q_1c}{d_1}}\right) \cdots\right).\\
	\end{align*}
	Now because  of $(c,x_1d_1) = ((c,q_1c), x_1d_1) 
	= (c, (q_1c, x_1d_1)) = (c,d_1)$, we deduce from Lemma \ref{inducedgauss2} that
	\[
		g(\chi^*,c,x_1d_1) = \overline{\chi^*(x_1)} g(\chi^*,c,d_1).
	\]
By Lemma \ref{inducedgauss2}, this Gauss sum is zero unless
	$c^*|\frac{c}{(c,d_1)}$,
	which implies the first divisibility condition of \eqref{divisibilitycondition2} because 
$c^*|\frac{c}{(c,d_1)} = \frac{d_1}{(c,d_1)}\frac{c}{d_1} | \frac{q_1c}{d_1}$ by \eqref{divisibilitycondition3}.

	Next we sum over $x_1$. Notice that $\overline{x_1}$ is its multiplicative inverse modulo $q_1c/d_1$ and 
	hence modulo $c^*$. This means that $\chi^*(\overline{x_1})= \overline{\chi^*(x_1)}$.
	We change variables in the $x_1$ summation $x_1 \to \overline{x_1}$, and change orders
	of summation to obtain
	\begin{align*}
		S&= g(\chi^*,c, d_1)\sums_{x_1 (\text{mod }{\frac{q_1c}{d_1}})} \overline{\chi^*(x_1)}
		\left(\sums_{x_2 (\text{mod }{\frac{q_1q_2c}{d_1d_2}})}
		e\left(\tfrac{d_2x_2 \overline{x_1}}{\frac{q_1c}{d_1}}\right) \cdots\right)\\
		&= g(\chi^*,c,d_1) \sums_{x_2 (\text{mod }{\frac{q_1q_2c}{d_1d_2}})} \,\,\,
		\sums_{x_1 (\text{mod }{\frac{q_1c}{d_1}})} 
		\chi^*(x_1) e\left(\tfrac{d_2x_2 x_1}{\frac{q_1c}{d_1}}\right) 
		\left(\sums_{x_3 (\text{mod }{\frac{q_1q_2q_3c}{d_1d_2d_3}})} 
		e\left(\tfrac{d_3x_3 \overline{x_2}}{\frac{q_1q_2c}{d_1d_2}}\right)\cdots \right)\\
		&= g(\chi^*,c,d_1) \sums_{x_2 (\text{mod }{\frac{q_1q_2c}{d_1d_2}})} g(\chi^*,\tfrac{q_1c}{d_1}, d_2x_2) \left(\sums_{x_3 (\text{mod }{\frac{q_1q_2q_3c}{d_1d_2d_3}})} 
		e\left(\tfrac{d_3x_3 \overline{x_2}}{\frac{q_1q_2c}{d_1d_2}}\right)\cdots \right).
	\end{align*}
	Once again the equalities $(\tfrac{q_1c}{d_1}, d_2x_2) = ((\tfrac{q_1c}{d_1},
	\tfrac{q_1q_2c}{d_1}), d_2x_2)=(\tfrac{q_1c}{d_1}, (\tfrac{q_1q_2c}{d_1}, d_2x_2))=
	(\tfrac{q_1c}{d_1}, d_2)$ imply that we can pull out $\chibs(x_2)$ from the Gauss sum. Then we have
	\[
		S = g(\chi^*,c,d_1)g(\chi^*,\tfrac{q_1c}{d_1}, d_2) \sums_{x_2 (\text{mod } {\frac{q_1q_2c}{d_1d_2}})} \overline{\chi^*(x_2)} \left(\sums_{x_3 (\text{mod }{\frac{q_1q_2q_3c}{d_1d_2d_3}})} 
		e\left(\tfrac{d_3x_3 \overline{x_2}}{\frac{q_1q_2c}{d_1d_2}}\right)\cdots \right).
	\]
	The second Gauss sum $g(\chi^*,\tfrac{q_1c}{d_1}, d_2)$ vanishes 
	 unless $c^*|\frac{q_1c/d_1}{(q_1c/d_1,d_2)}$	by Lemma \ref{inducedgauss2}. This in turn implies 
	 $c^*|\frac{q_1c/d_1}{(q_1c/d_1,d_2)}|\frac{cq_1q_2}{d_1d_2}$ by \reff{divisibilitycondition3}, which is the second divisibility condition of \reff{divisibilitycondition2}. We complete the proof after repeating this process  $(N-2)$ times.  
\end{proof}



\begin{proposition}\label{equivalent}
The equations \eqref{eq:glnVoronoi} and \eqref{eq:HequalsG} are equivalent.
\end{proposition}

\begin{proof}
Let $\chi$ be a Dirichlet character modulo $c$ induced from the primitive Dirichlet character $\chis$ modulo $c^*$.
Multiply both sides of \eqref{eq:glnVoronoi} by $\chi(a)$ and sum over reduced residue classes modulo $c$.
On the left hand side of \eqref{eq:glnVoronoi},  one gets $$\sum_{\substack{a\mod c\\(a,c)=1}}
\chi(a)L_{\vec{q}}(s,F,a/c)=
(c/c^*)^{1-2s}H(\vec{q}; c, \chi^*, s),$$ whereas on the right hand side of \eqref{eq:glnVoronoi}, one obtains $(c/c^*)^{1-2s} G(\vec{q}; c,\chi^*,s)$ by making use of Lemma \ref{lem:averageGLnKloosterman} and the fact that
	\[
		g(\chi^*,\tfrac{q_1\cdots q_{N-2}c}{d_1\cdots d_{N-2}},-n) = \pm g(\chi^*,\tfrac{q_1\cdots q_{N-2}c}{d_1\cdots d_{N-2}},n),
	\]
	depending on whether $\chi(-1)$ is $1$ or $-1$. This shows that  \eqref{eq:glnVoronoi} implies \eqref{eq:HequalsG}.
Conversely if we multiply both sides of \eqref{eq:HequalsG}	
by $\frac{1}{\phi(c)}\overline{\chi(a)}$ and sum over all Dirichlet characters (both primitive and non-primitive)
modulo $c$, we obtain \eqref{eq:glnVoronoi}, by using the orthogonality relation for Dirichlet characters.
Since both of the aforementioned summations that shuttle between 
 \eqref{eq:glnVoronoi} and \eqref{eq:HequalsG} 
are finite,
the properties of absolute convergence and analytic continuation are preserved.     
\end{proof}

\section*{Acknowledgements}
The authors would like to thank Dorian Goldfeld, Wenzhi Luo for helpful suggestions, Matthew Young for his extensive help in organization of the manuscript and his encouragement, and Jeffrey Hoffstein, in whose class the seed for this work originated.

\bibliographystyle{alpha}
\bibliography{bibib}

$$$$
\noindent {\scshape{Eren Mehmet K{\i}ral   }} \\
Department of Mathematics\\
Texas A\&M University\\
College Station, TX 77843, USA \\
ekiral@math.tamu.edu
$$\;$$
\noindent {\scshape{Fan Zhou}} \\
Department of Mathematics\\
The Ohio State University\\
Columbus, OH 43210, USA\\
zhou.1406@math.osu.edu

\end{document}